\documentclass[11pt]{amsart}

\usepackage{amsmath,amssymb,latexsym,amsthm,newlfont,enumerate}
\usepackage{hyperref}

\usepackage[T1]{fontenc}
\usepackage[all]{xy}

\theoremstyle{plain}

\newtheorem{thm}{Theorem}[section]

\newtheorem{pro}[thm]{Proposition}
\newtheorem{prb}[thm]{Problem}
\newtheorem{lem}[thm]{Lemma}

\newtheorem{cor}[thm]{Corollary}
\newtheorem{con}[thm]{Conjecture}

\newtheorem{thmA}{Theorem}

\newtheorem{proA}[thmA]{Proposition}

\theoremstyle{definition}

\newtheorem{dfn}[thm]{Definition}

\newtheorem{nt}[thm]{Notation}

\newtheorem{rem}[thm]{Remark}

\newtheorem{exa}[thm]{Example}

\theoremstyle{remark}

\newcommand{\Z}{\mathbb{Z}}

\newcommand{\C}{\mathbb{C}}
\newcommand{\R}{\mathbb{R}}
\newcommand{\Q}{\mathbb{Q}}
\newcommand{\PS}{\mathbb{P}}

\newcommand{\OO}{\mathcal{O}}

\newcommand{\reg}{\mathrm{reg}}

\DeclareMathOperator{\rk}{rk}

\DeclareMathOperator{\codim}{codim}

\DeclareMathOperator{\Sing}{Sing}
\DeclareMathOperator{\Exc}{Exc}

\DeclareMathOperator{\Pic}{Pic}

\DeclareMathOperator{\NEb}{\overline{\mathrm{NE}}}

\DeclareMathOperator{\NMb}{\overline{\mathrm{NM}}}

\DeclareMathOperator{\pr}{\mathrm{pr}}

\newcommand\sF{{\mathcal F}}

\newcommand\sL{{\mathcal L}}

\newcommand\sO{{\mathcal O}}

\newcommand\sC{{\mathcal C}}

\title[Connectedness problems]{Maps from K-trivial varieties \\ and connectedness problems} 
\author{Vladimir Lazi\'c}
\address{Fachrichtung Mathematik, Campus, Geb\"aude E2.4, Universit\"at des Saarlandes, 66123 Saarbr\"ucken, Germany}
\email{lazic@math.uni-sb.de}

\author{Thomas Peternell}
\address{Thomas Peternell, Mathematisches Institut, Universit\"at Bayreuth, 95440 Bayreuth, 
Germany}
\email{thomas.peternell@uni-bayreuth.de}

\thanks{
Lazi\'c was supported by the DFG-Emmy-Noether-Nachwuchsgruppe ``Gute Strukturen in der h\"oherdimensionalen birationalen Geometrie''. Peternell was supported by the DFG grant ``Zur Positivit\"at in der komplexen Geometrie''. We would like to thank F.-O.\ Schreyer for useful discussions.
\newline
\indent 2010 \emph{Mathematics Subject Classification}: 14J32, 14H10.\newline
\indent \emph{Keywords}: Calabi-Yau varieties, families of elliptic curves, elliptically (chain) connected varieties.
}

\begin{document}

\begin{abstract} 
In this paper we study varieties covered by rational or elliptic curves. First, we show that images of Calabi-Yau or irreducible symplectic varieties under rational maps are almost always rationally connected. Second, we investigate elliptically connected and elliptically chain connected varieties, and varieties swept out by a family of elliptic curves. Among other things, we show that Calabi-Yau or hyperk\"ahler manifolds which are covered by a family of elliptic curves contain uniruled divisors and that elliptically chain connected varieties of dimension at least $2$ contain a rational curve, and so do $K$-trivial varieties with finite fundamental group which are covered by elliptic curves.
\end{abstract}

\maketitle
\setcounter{tocdepth}{1}
\tableofcontents

\section{Introduction} 

There are two main themes in this paper. First, starting from a morphism (or merely a dominant rational map) from a Calabi-Yau manifold, it is a natural and important question to determine the geometry of the target variety. Indeed, F.\ Bogomolov raised this question at a conference in B\k{e}dlewo in June 2016, which was the starting point of this paper. Throughout the paper we work over $\C$.

The answer is given in the following result.

\begin{thmA} \label{thm:main1}
Let $X$ be a Calabi-Yau variety or an irreducible symplectic variety. Let $f\colon X \dasharrow Y$ be a dominant rational map to a projective manifold $Y$. Suppose that either $\dim Y < \dim X$ or that $\dim X = \dim Y$ and $f$ is ramified in codimension $1$. Then $Y$ is rationally connected. 
\end{thmA} 

For the notions of singular Calabi-Yau and singular irreducible symplectic varieties we refer to Section \ref{sec:prelim}. The proof of this result is in Section \ref{sec:ratconn}.

\medskip

Second, we study elliptically chain connected varieties. If $\mathcal C = (C_t)_{t\in T}$ is  family of elliptic curves covering a projective variety $X$, we say that $X$ is \emph{elliptically chain connected} with respect to $\mathcal C$ if every two very general points on $X$ can be joined by a finite connected chain of curves $C_t$. Similarly, the variety $X$ is \emph{elliptically connected} if there exists a family $(C_t)_{t\in T}$ of elliptic curves such that two general points of $X$ can be joined by one curve $C_t$. These notions are analogous to rational chain connectedness (for some family of rational curves) and rational connectedness, and have been first studied in detail by Gounelas \cite{Gou16}. If $X$ is smooth, rational connectedness and rational chain connectedness coincide by \cite[2.1]{KMM92}. However, in the elliptic setting, things change drastically: elliptic connectedness is much stronger than elliptic chain connectedness. For instance, we show:

\begin{thmA} 
Let $X$ be a Calabi-Yau threefold such that there is a finite map $f\colon X \to Y$ to a smooth cubic threefold $Y \subseteq \mathbb P^4$. Then $X$ is elliptically chain connected but not elliptically connected.
\end{thmA} 

This result follows from Theorem \ref{thm:map} and Example \ref{ex}. The statement that $X$ is not elliptically connected stems from \cite[Theorem 6.2]{Gou16}, stating that otherwise $X$ would be rationally connected or would admit a rationally connected fibration to an elliptic curve. There are actually many more examples of elliptically chain connected but not elliptically connected varieties;
in principle, any non-uniruled manifold which is elliptically chain connected provides an example. 

\medskip

We next study the structure of elliptically chain connected varieties which are not uniruled. The starting point is the following result, proved in Section \ref{sec:ellchainconn}.

\begin{proA} 
Let $X$ be a normal projective $\mathbb Q$-Gorenstein variety and $\mathcal C = (C_t)_{t\in T} $ be a covering family of elliptic curves such that $X$ is smooth along the curve $C_t$ for general $t\in T$. Then:
\begin{enumerate} 
\item[(i)] $K_X \cdot C_t \leq 0$ for all $t$;
\item[(ii)] if $K_X$ is pseudoeffective, then the normalisation morphism of a general curve $C_t$ is \'etale and $K_X \cdot C_t = 0$ for all $t$;
\item[(iii)] if $K_X$ is pseudoeffective and if for a general $t\in T$ the curve $C_t$ is smooth, then its normal bundle $N_{C_t/X}$ is trivial.
\end{enumerate} 
\end{proA} 

When $X$ has canonical singularities, then assertion (i) and the second part of assertion (ii) hold without assuming that $X$ is smooth near the general curve $C_t$, see Corollary \ref{cor:can}.

It is natural to expect that elliptically chain connected manifolds $X$ should have non-positive Kodaira dimension, and we prove this provided $X$ has a minimal model reached only by a sequence of divisorial contractions, without any flips. 
However, the general case remains mysterious, even assuming the Minimal Model Program, see Remarks \ref{rem1} and \ref{rem3}. 

\medskip

Finally, in Section \ref{sec:elltorat} we study the problem of existence of rational curves, provided the variety is already covered by elliptic curves; see also \cite{DFM16} for the case of elliptic and abelian fibrations on Calabi-Yau manifolds. We first show in Theorem \ref{thm:CYHK}:

\begin{thmA} 
Let $X$ be a Calabi-Yau or hyperk\"ahler manifold of dimension $\dim X \geq 2$.  If $X$ has a covering family of elliptic curves, then $X$ contains a uniruled divisor. In particular, $X$ contains a rational curve.  
\end{thmA} 

If $X$ is moreover elliptically chain connected, we obtain in Theorem \ref{thm:ellipticrational}, without any further assumption on $X$:

\begin{thmA} 
Let $X$ be an elliptically chain connected projective manifold with $\dim X \geq 2$. Then $X$ contains a rational curve. 
\end{thmA} 

In Theorem \ref{thm:ratcurves} we obtain a precise structure theorem in the absence of rational curves.

\begin{thmA}  
Let $X$ be a projective manifold of dimension at least $2$ with a covering family $\sC$ of elliptic curves. Suppose that $X$ does not contain rational curves. Then there exists a fibration $\varphi\colon X \to W$ to a normal projective variety $W$ with the following properties:
\begin{enumerate} 
\item[(i)] $\varphi$ contract all elements of $\mathcal C$; more precisely, $\varphi$ is the $\mathcal C$-quotient of $X$;
\item[(ii)] all fibres of $\varphi$ are irreducible;
\item[(iii)] the normalisation of any fibre of $\varphi$ is an elliptic curve;
\item[(iv)] $\varphi$ is an almost smooth elliptic fibration over the smooth locus of $W$;
\item[(v)] $W$ has klt singularities.
\end{enumerate} 
\end{thmA} 

We end the paper with a list of open problems and conjectures.

\section{Preliminaries} \label{sec:prelim}

\begin{nt} 
Let $X$ be a normal projective variety. By $\Omega^{[p]}_X$ we denote the sheaf of reflexive $p$-differentials. In other words, if $j\colon X_{\reg} \to X$ denotes the inclusion of the smooth locus, then 
$$ \Omega^{[p]}_X = j_*\Omega^p_{X_{\reg}}. $$
If $\mathcal F$ is a coherent sheaf on $X$, then for every positive integer $m$, $\mathcal F^{[m]}$ denotes the reflexive tensor power $(\mathcal F^{\otimes m})^{**}$. For a proper morphism $f\colon Y\to X$, we have the reflexive pullback $f^{[*]}\mathcal F:=(f^*\mathcal F)^{**}$.
\end{nt} 

Following \cite{GKP16}, we define singular Calabi-Yau varieties and singular irreducible symplectic varieties as follows. 

\begin{dfn}\label{def:CYSympl}
Let $X$ be a normal projective variety with trivial canonical sheaf $\omega_X \simeq \sO_X$ having canonical singularities.
\begin{enumerate}
\item[(a)] We call $X$ \emph{Calabi-Yau} if 
$$H^0 \big( \widetilde X, \Omega_{\widetilde X}^{[q]} \big ) = 0$$ 
for all $0 < q < \dim X$ and for all finite covers $\widetilde X \to X$ which are \'etale in codimension one.
\item[(b)] We call $X$ \emph{irreducible (holomorphic) symplectic} if there exists a reflexive $2$-form 
$$\sigma \in H^0\big(X, \Omega_{X}^{[2]} \big)$$ 
such that $\sigma$ is everywhere non-degenerate on $X_{\reg}$, and such that for all finite covers $f\colon \widetilde X \to X$ which are \'etale in codimension one, the exterior algebra of global reflexive forms is generated by $f^*\sigma$.
\end{enumerate}
\end{dfn}

When $X$ is smooth, we get back the classical definitions; the fundamental group is automatically finite in the smooth case (which is unclear in the singular setting). 
Note that sometimes a Calabi-Yau manifold or an irreducible symplectic manifold $X$ is assumed to be simply connected. When $X$ has even dimension, finiteness of $\pi_1(X)$ is automatic in our setting by \cite[Proposition 8.23]{GKP16}. 

\begin{rem}  \label{rem4}  
Let $X$ be a normal projective variety with trivial canonical sheaf and with canonical singularities, and let $f\colon \widetilde X \to X$ be a finite cover which is \'etale in codimension one. Since 
$$ H^0\big(\widetilde X,\Omega^{[q]}_{\widetilde X}\big) \simeq H^q\big(\widetilde X,\OO_{\widetilde X}\big) $$
by \cite[Proposition 6.9]{GKP16},  we obtain the vanishing
$$ H^q\big(\widetilde X,\OO_{\widetilde X}\big) = 0 $$
for all $1 \leq q \leq \dim X -1$ in the Calabi-Yau case, respectively for\ all odd numbers $1 \leq q  \leq \dim X -1$ in the symplectic case.
\end{rem}

We need the notion of a pseudoeffective reflexive sheaf of rank $1$, see \cite[Definition 2.1]{HP17} for the definition in the higher rank case. 

\begin{dfn} \label{def:ps} 
Let $X$ be a normal projective variety and $\mathcal L$ a reflexive sheaf of rank $1$ on $X$. Then $\mathcal L$ is \emph{pseudoeffective} if there is an ample divisor $H$ on $X$ such that for all $c > 0$ there are positive integers $i$ and $j$ with $i > cj$ such that
$$ H^0\big(X, \mathcal L^{[i]} \otimes \OO_X(jH)\big) \neq 0.$$
\end{dfn} 

\begin{rem} 
Let $\mathcal L$ be a reflexive sheaf of rank $1$. Then $\mathcal L$ is pseudoeffective if and only if some reflexive power $\mathcal L^{[m]} $ is pseudoeffective.

Assume now that $\mathcal L$ is a $\mathbb Q$-line bundle, i.e.\ some reflexive power $\mathcal L^{[m]} $ is locally free. If $D$ is a Weil divisor associated to $\mathcal L$, then $D$ is $\Q$-Cartier. In this setting, $\mathcal L$ is pseudoeffective if and only the $\Q$-Cartier divisor $D$ is pseudoeffective.
\end{rem} 

\begin{lem} \label{lem1} 
Let $X$ be a $\Q$-factorial projective variety of dimension $n$ with canonical singularities and trivial canonical sheaf, and let $ \sL \subseteq \Omega^{[p]}_X $ be a pseudoeffective reflexive 
subsheaf of rank $1$ for some $p > 0$. Then:
\begin{enumerate}
\item[(i)] $\mathcal L$ is numerically trivial;
\item[(ii)] if $X$ is a Calabi-Yau variety, then $p = n$ and there exists a quasi-\'etale cover $f\colon \widetilde X \to X$  such that  $f^{[*]}\sL \simeq \OO_{\widetilde X}$;
\item[(iii)] if $X$ is an irreducible symplectic variety with a symplectic form $\sigma$,  then $p$ is even and there exists a quasi-\'etale cover $f\colon \widetilde X \to X$  such that  $f^{[*]}\sL \simeq \OO_{\widetilde X} \subseteq \Omega^{[p]}_{\widetilde X}$. Furthermore, the form defined by the inclusion  $\OO_{\widetilde X} \subseteq \Omega^{[p]}_{\widetilde X}$ is a scalar multiple of $f^*(\sigma^{p/2})$.
\end{enumerate} 
\end{lem} 

\begin{proof} 
If $H$ is an ample divisor on $X$, then the tangent sheaf $\mathcal T_X$ is  $H$-semistable by \cite[Proposition 5.4]{GKP16}. Since $K_X\sim0$, we have $\mathcal T_X\simeq\Omega_X^{[n-1]}$, and then it is easily seen that $\Omega^{[p]}_X$ is $H$-semistable. Consequently, $c_1(\sL) \cdot H^{n-1} \leq 0$. Since $\mathcal L$ is pseudoeffective and $X$ is $\Q$-factorial, this shows (i). 

If now $X$ is a Calabi-Yau or an irreducible symplectic variety, then $\sL$ is torsion by (i) and by Remark \ref{rem4}, and it defines a quasi-\'etale cover $f\colon \widetilde X \to X$ such that $f^{[*]}\sL\simeq\OO_{\widetilde X}$, see \cite[Lemma 2.53]{KM98}. Furthermore,
\begin{equation}\label{eq:30}
f^{[*]}\sL \subseteq f^{[*]}\Omega^{[p]}_X \subseteq \Omega^{[p]}_{\widetilde X}.
\end{equation}
When $X$ is a Calabi-Yau, this forces $p = n$, which shows (ii).

When $X$ is an irreducible symplectic variety, \eqref{eq:30} yields a holomorphic reflexive $p$-form $\omega $ on $\widetilde X$. Then (iii) is immediate by the definition of irreducible symplectic varieties.
\end{proof}

For later use we note the following well-known result.

\begin{lem}\label{lem4}
Let $f\colon X\to Y$ be a proper morphism with connected fibres between normal varieties. Let $\pi\colon Y'\to Y$ be a proper birational morphism from a normal variety $Y'$ and let $X'$ be the normalisation of the main component of the fibre product $Y'\times_Y X$ with the commutative diagram
$$
\xymatrix{
X' \ar[d]_{\pi'} \ar[r]^{f'} & Y'  \ar[d]^\pi \\
X \ar[r]_f & Y. &
}
$$
Then the morphism $f'$ has connected fibres.
\end{lem}

\begin{proof}
Let 
$$ X' \buildrel {\alpha} \over {\longrightarrow } \widehat Y  \buildrel {\beta} \over {\longrightarrow} Y' $$
be the Stein factorisation of $f'$. By the proof of \cite[Lemma 4.1.14]{Laz04} there is a torsion free sheaf $\mathcal F$ on $Y$ such that $\beta_*\OO_{\widehat{Y}}=\OO_Y\oplus\mathcal F$, hence
\begin{equation}\label{eq:50}
f'_*\OO_{X'}=\OO_Y\oplus\mathcal F.
\end{equation}
Now, since $\pi'$ has connected fibres by Zariski's main theorem, so does $f\circ\pi'=\pi\circ f'$. Pushing forward \eqref{eq:50} by $\pi$ we obtain $\pi_*\mathcal F=0$, hence $\mathcal F=0$. Therefore, $\beta$ is an isomorphism by Zariski's main theorem.
\end{proof}

We also need the following consequence of a more general version of the Negativity lemma.

\begin{lem}\label{lem:exceptional}
Let $f\colon X\to Y$ be a birational morphism from a projective manifold $X$ to a $\Q$-factorial projective variety $Y$. Let $\lambda_i$ be real numbers and let $E_i$ be $f$-exceptional prime divisors on $X$ such that the divisor $\sum \lambda_i E_i$ is pseudoeffective. Then $\lambda_i\geq0$ for all $i$.
\end{lem}

\begin{proof}
Consider Nakayama's divisorial Zariski decomposition as in  \cite[Chapter III]{Nak04}: denote $P:=P_\sigma(\sum \lambda_i E_i)$ and $N_\sigma(\sum \lambda_i E_i)=\sum \alpha_j D_j$ for prime divisors $D_j$ and real numbers $\alpha_j\geq0$. Then pushing forward the relation 
$$\sum \lambda_i E_i=P+\sum \alpha_j D_j$$
to $Y$ via $f$ implies that all $D_i$ are $f$-exceptional and $f_*P\sim_\R 0$; we may assume that $D_i=E_i$, and hence 
$$P=\sum(\lambda_i-\alpha_i)E_i.$$ 
Now applying \cite[Lemma 3.3]{Bir12} for $D:=-P$ implies that $\lambda_i\leq\alpha_i$, hence $\lambda_i=\alpha_i$ since $P$ is pseudoeffective. This concludes the proof.
\end{proof}

\section{Rational connectedness}\label{sec:ratconn}

In this section we prove Theorem \ref{thm:main1}. We first consider maps to lower-dimensional varieties and start with the following lemma. For the definition of a $\Q$-factorialisation, see \cite[Corollary 1.37]{Kol13}.

\begin{lem} \label{lem:reduction} 
Let $X$ be normal projective klt variety of dimension $n$ with $\omega_X \simeq \OO_X$. Let $\varphi\colon \widetilde X \to X$ be a $\Q$-factorialisation. If $X$ is a Calabi-Yau, respectively irreducible symplectic, variety, then so is $\widetilde X$. 
\end{lem} 

\begin{proof}
Since $\varphi$ is an isomorphism in codimension $1$, we have $\omega_{\widetilde X}\simeq \OO_{\widetilde X}$, and $\widetilde X$ has klt singularities by \cite[Proposition 5.20]{KM98}. 

Assume first that $X$ is a Calabi-Yau variety. Let $\psi\colon \widehat X \to \widetilde X$ be a quasi-\'etale map. We need to show that 
\begin{equation} \label{eq:lohengrin}  
H^0\big(\widehat X,\Omega^{[q]}_{\widehat X}\big) = 0 
\end{equation} 
for $1 \leq q \leq n-1$. 
Let $$\widehat X \buildrel  {\alpha} \over {\longrightarrow} Z \buildrel {\beta} \over {\longrightarrow} X$$
be the Stein factorization of $\varphi\circ\psi$. Then $\alpha$ is birational, $\beta $ is quasi-\'etale and $Z$ has klt singularities again by \cite[Proposition 5.20]{KM98}. Since $X$ is Calabi-Yau, we have
$$ H^0\big(Z,\Omega^{[q]}_Z\big) = 0.$$ 
Then $H^0\big(\widehat X,\Omega^{[q]}_{\widehat X}\big) \simeq H^0\big(Z,\Omega^{[q]}_Z\big)$ by \cite[Theorem 1.4]{GKKP11}, we obtain (\ref{eq:lohengrin}). 

Assume now that $X$ is an irreducible symplectic variety, and let $\omega \in H^0\big(X,\Omega^{[2]}_X\big) $ be the symplectic form. Let $\psi\colon \widehat X \to \widetilde X$ be a quasi-\'etale map. If $q$ is odd, the very same argument as above shows that 
$$ H^0\big(\widehat X,\Omega^{[q]}_{\widehat X}\big) = 0.$$ 
Let now $q $ be even and let $\eta \in H^0\big(\widehat X,\Omega^{[q]}_{\widehat X}\big)$. Using the Stein factorisation as above, we observe that $Z$ is irreducible symplectic, with the symplectic form $\omega_Z = \beta^*(\omega)$. Again by \cite[Theorem 1.4]{GKKP11}, we find $\lambda \in \C$ such that $\eta = \lambda\alpha^*(\omega_Z)$. Since $\alpha^*(\omega_Z)$ is non-degenerate on $\widehat X_{\reg}$, we conclude.
\end{proof} 

\begin{thm} \label{thm1-1} 
Let $X$ be a Calabi-Yau variety or an irreducible symplectic variety. Let $f\colon X \dasharrow Y$ be a dominant rational map to a projective manifold $Y$. If $\dim Y < \dim X$, then $Y$ is rationally connected. 
\end{thm}

\begin{proof}
By passing to a $\Q$-factorialisation, by Lemma \ref{lem:reduction} we may assume that $X$ is $\Q$-factorial. Choose a resolution of indeterminacies $\pi\colon X' \to X$ of $f$ with $X'$ smooth and let $f'\colon X' \to Y$ be the induced morphism. Consider the Stein factorisation $g\colon X' \to Y'$ of $f'$. It suffices to show that $Y'$ is rationally connected. By blowing up $X'$ and $Y'$ further, by Lemma \ref{lem4} we may assume that $Y'$ is smooth. Hence, by replacing $Y$ by $Y'$ and $f'$ by $g$, we may assume from the beginning that $f'$ has connected fibres. 

We claim that $Y$ is uniruled. This immediately implies the result: indeed, let $h\colon Y \dasharrow Z$ be an MRC fibration of $Y$ such that $Z$ is smooth, see \cite[Section IV.5]{Kol96}. If $Y$ is not rationally connected, then $\dim Z > 0$ and $Z$ is not uniruled by \cite{GHS03}. But then the claim applied to the dominant rational map $h\circ f\colon X\dashrightarrow Z$ implies that $Z$ is uniruled, a contradiction.

It remains to prove the claim. Assume to the contrary that $Y$ is not uniruled and denote $d = \dim Y$. Then $K_Y$ is pseudoeffective by \cite[Corollary 0.3]{BDPP}. Set
$$\sL' := f'^*\Omega_Y^d.$$
From the canonical injective morphism $f'^*\Omega_Y^d\to \Omega^d_{X'}$ coming from the cotangent sequence we obtain a pseudoeffective rank $1$ reflexive sheaf 
$$\mathcal L=(\pi_*\sL')^{**}\subseteq \Omega_X^{[d]}.$$

This already gives a contradiction if $X$ is a Calabi-Yau variety by Lemma \ref{lem1}(ii) since $d < \dim X$. 

Therefore, for the rest of the proof we assume that $X$ is an irreducible symplectic variety with the symplectic form $\sigma$. By Lemma \ref{lem1}(iii) there exists a quasi-\'etale cover $g\colon\widetilde X \to X$ such that 
\begin{equation}\label{eq:sim}
\OO_{\widetilde X} \simeq g^{[*]}\mathcal L\subseteq \Omega^{[d]}_{\widetilde X}.
\end{equation}
By replacing $\widetilde{X}$ by its $\Q$-factorialisation, by Lemma \ref{lem:reduction} we may assume that $\widetilde{X}$ is a $\Q$-factorial irreducible symplectic variety. 

Let $\widehat X$ be a desingularisation of the main component of the fibre product $X' \times_X \widetilde X$ with induced morphisms $\tau\colon \widehat X \to  X'$ and $g'\colon \widehat X \to X'$; note that $\tau $ is birational. 
Consider the Stein factorization
$$ \widehat X \buildrel {\alpha} \over {\longrightarrow } \widehat Y  \buildrel {\beta} \over {\longrightarrow} Y $$
of $f' \circ g'$. Blowing up $\widehat X$ and $\widehat Y$ further, by Lemma \ref{lem4} we may assume $\widehat Y$ smooth; then the map $\beta$ becomes generically finite.
$$
\xymatrix{
 & & \widehat Y \ar[d]^\beta \\
\widehat X\ar[d]_{\tau} \ar[r]_{g'} \ar[urr]^\alpha & X'  \ar[d]_\pi \ar[r]^{f'} & Y \\
\widetilde X \ar[r]_g & X \ar@{-->}[ur]_f, &
}
$$
Let 
$$ \widehat {\mathcal L} := g'^*\mathcal L',$$
and let $X^\circ\subseteq X$ be an open subset with $\codim_X(X\setminus X^\circ)\geq2$ over which $g$ is \'etale. Then by flat base change we have an isomorphism $ (g^*\mathcal L)|_{X^\circ} \simeq \big(\tau_*\widehat  {\mathcal L}\big)|_{X^\circ}$, and it extends to an isomorphism 
$$ g^{[*]}\mathcal L \simeq \big(\tau_*\widehat  {\mathcal L}\big)^{**}.$$
This and \eqref{eq:sim} imply
$$ \big(\tau_*\widehat {\mathcal L}\big)^{**} \simeq \OO_{\widetilde X},$$ 
thus there exist integers $\lambda_i $ and $\tau$-exceptional prime divisors $E_i$ on $\widehat  X$ such that
$$\textstyle \widehat {\mathcal L} \simeq \OO_{\widehat X}\big(\sum \lambda_i E_i\big).$$
Then Lemma \ref{lem:exceptional} implies that $\lambda_i\geq0$ for all $i$ and therefore
$$ h^0\big(\widehat X, \widehat {\mathcal L}\big) = 1.$$ 
Since $ \widehat {\mathcal L} = \alpha^* \beta^*\Omega_Y^d$ and since $\alpha$ has connected fibres, we get
$$h^0\big(\widehat Y,\beta^*\Omega_Y^d\big) =1.$$
Since we have an inclusion $\beta^*\Omega_Y^d\subseteq\Omega_{\widehat{Y}}^d$, it follows finally that
\begin{equation}\label{eq:40}
h^0\big(\widehat Y,\Omega_{\widehat{Y}}^d\big) \ne 0. 
\end{equation}
By \cite[Theorem 1.4]{GKKP11} we have $ H^0\big(\widehat X,\Omega^d_{\widehat X}\big) \simeq H^0\big(\widetilde X, \Omega^{[d]}_{\widetilde X}\big)$, hence 
\begin{equation}\label{eq:41}
h^0\big(\widehat X,\Omega^d_{\widehat X}\big) = 1
\end{equation}
as $\widetilde X$ is irreducible symplectic. By \eqref{eq:40} we can pick a nonzero form $ \eta \in H^0\big(\widehat Y,\Omega^d_{\widehat Y}\big)$, hence by \eqref{eq:41} then there exists $\lambda \in \C\setminus\{0\}$ such that 
$$ \alpha^*(\eta) = \lambda \tau^* g^*(\sigma^{d/2}).$$
The form $ \tau^* g^*(\sigma^{d/2})$ is therefore degenerate along the fibres of $\alpha$, and since $\dim Y < \dim X$, this is a contradiction which finishes the proof. 
\end{proof} 

We next discuss the case $\dim Y = \dim X.$ 
We say that a dominant rational map $f\colon X\dashrightarrow Y$ between normal projective varieties has \emph{ramification in codimension $1$} if there exists an irreducible analytic set $A \subseteq  Y$ of codimension $1$ such that:
\begin{enumerate}
\item[(i)] there exists a resolution of indeterminacies $\pi\colon X' \to X$ of $f$ such that the exceptional set of $\pi$ does not project onto $A$; 
\item[(ii)] for any $a \in A$ there exists $b \in f^{-1}(a)$ such that the Jacobian of $f$ does not have maximal rank at $b$. 
\end{enumerate} 

In this terminology we have:

\begin{thm} \label{thm1-2} 
Let $X$ be a Calabi-Yau variety or an irreducible symplectic variety. Let $f\colon X \dasharrow Y$ be a dominant rational map to a normal projective variety $Y$ with  canonical singularities such that $\dim X = \dim Y$, and assume that $f$ has ramification in codimension $1$. Then $Y$ is rationally connected.
\end{thm} 

The case when $Y$ is the quotient of $X$ by a finite group has been studied in detail in \cite{KL09}. 

\begin{proof}   
It suffices to show that $Y$ is uniruled: indeed, then we may consider an MRC fibration $Y\dashrightarrow Z$ of $Y$, and we have $\dim Z<\dim Y$. Theorem \ref{thm1-1} implies that $Z$ must be rationally connected, which contradicts the main result of \cite{GHS03} if $\dim Z>0$. Therefore, $Z$ must be a point, hence $Y$ is rationally connected.
 
It remains to show the claim. We may assume that $Y$ is smooth. By passing to a $\Q$-factorialisation, by Lemma \ref{lem:reduction} we may assume that $X$ is $\Q$-factorial.  Arguing by contradiction, assume that $Y$ is not uniruled, and hence that $K_Y$ is pseudoeffective by \cite[Corollary 0.3]{BDPP}. Set $n = \dim X$, choose a resolution of indeterminacies $\pi\colon X' \to X$ of $f$ with $X'$ smooth and let $f'\colon X' \to Y$ be the induced morphism. Since $f'^*\Omega^n_Y \subseteq \Omega^n_{X'}$, we may write
\begin{equation}\label{eq:1}
K_{X'} \sim f'^*K_Y + \sum a_i E_i,
\end{equation}
where the $E_i $ are the irreducible components of the ramification divisor, and $a_i>0$ for all $i$. The divisor $L := \pi_*f'^*K_Y $ is pseudoeffective and 
$$\textstyle E:=\pi_*\big(\sum a_i E_i\big)$$ 
is a non-zero effective divisor, since $f$ is ramified in codimension $1$. Pushing forward the relation \eqref{eq:1} to $X$ by $\pi$ we obtain
$$ 0 \sim_\Q L + E, $$
a contradiction which finishes the proof.
\end{proof} 

\emph{To sum up:} we may basically say that given a rational dominant map $f\colon X \dasharrow Y $ from a Calabi-Yau variety or an irreducible symplectic variety $X$ to a projective manifold $Y$, then $Y$ is rationally connected unless $f$ is a composition of a birational and an \'etale map. 

\begin{rem}
One cannot expect $Y$ to be rational, see Example \ref{ex}.
\end{rem} 

\section{Maps from $K$-trivial varieties to Fano manifolds}

We recall the notion of a quotient of a normal projective variety with respect to a covering family and refer to \cite{Ca81,Ca04,BCE+} for the details. 

\begin{dfn}
A \emph{covering family} $(C_t)_{t \in T}$ of $X$ by varieties of dimension $d$ is given by its \emph{graph}, respectively \emph{projection}
$$ p\colon \mathcal C \to X,\quad\text{respectively}\quad q\colon \mathcal C \to T,$$ 
where:
\begin{enumerate}
\item[(a)] $T$ and  $\mathcal C \subseteq X \times T$ are irreducible reduced projective varieties (with projections $p = \pr_X|_\mathcal{C}$ and $q = \pr_T|_\mathcal{C}$);  
\item[(b)] $p$ is surjective and $q$ is equidimensional;
\item[(c)]  the cycle-theoretic preimage $C_t$ of $t \in T$ is a purely $d$-dimensional cycle in $X$. 
\end{enumerate}
By abuse of notation we also say that $\mathcal C$ is a covering family. 
\end{dfn}

For details, we refer to \cite[\S I.4]{Kol96}; for the analytic case see \cite{Bar75,BM14}. 

We can form the quotient of $X$ with respect to this family: 

\begin{pro} \label{quot} 
Let $X$ be a normal projective variety and let $(C_t)_{t \in T}$ be a covering family of $X$. Then there is an almost holomorphic map 
$$ f\colon X \dasharrow Z$$
to a projective manifold $Z$ such that for two very general points $x_1,x_2 \in X$ we have $f(x_1) = f(x_2)$ if and only if $x_1 $ and $x_2$ can be joined by a chain of varieties $C_t$. The map $f$ is \emph{the quotient of $X$ with respect to the family $(C_t)_{t\in T}$}. We say that $f$, or $Z$, is the \emph{$\mathcal C$-quotient of $X$}. 
\end{pro} 

\begin{proof}  
See \cite{Ca81,Ca04}.
\end{proof} 

\begin{dfn} 
Let $X$ be a normal projective variety and let $\mathcal C$ be a covering family of $X$. Then:
\begin{enumerate} 
\item[(a)] $X$ is \emph{$\mathcal C$-chain connected} if its $\mathcal C$-quotient is a point;
\item[(b)] $X$ is \emph{elliptically chain connected} if it is $\mathcal C$-chain connected for some family $\mathcal C$ such that the normalisation of the  general member is an elliptic curve;
\item[(c)] $X$ is \emph{torically chain connected} if it is $\mathcal C$-chain connected for some family $\mathcal C$ whose general member is birational to an abelian variety.
\end{enumerate} 
\end{dfn} 

\begin{lem} \label{lem:triv} 
Let $X$ be a $\Q$-factorial projective variety with $\rho(X) = 1$ and let $\mathcal C$ be a covering family of $X$. Then $X$ is $\mathcal C$-chain connected.
\end{lem}

\begin{proof} 
Consider the $\mathcal C$-quotient $f\colon X \dasharrow Z$ and assume that $\dim Z > 0.$ Since the map $f$ is almost holomorphic, we may choose a resolution of indeterminacies $\pi\colon X' \to X$ of $f$ 
which does not alter a general fibre of $f$, and let $f'\colon X' \to Y$ be the induced morphism. Let  $D$ be a general ample divisor on $Z$. Then $D_0 = \pi_*f'^*D$ is an effective \emph{non-zero} $\Q$-Cartier divisor on $X$. Since $D_0 \cdot C = 0 $ for each curve $C$ which is contracted by the almost holomorphic map $f$, the divisor $D_0$ cannot be ample, and this contradicts the assumption on the Picard number of $X$.
\end{proof} 

The following is an analogue of rational connectedness in the case of elliptic curves \cite[Definition 3.3]{Gou16}. 

\begin{dfn} 
A projective manifold $X$ is said to be \emph{elliptically connected} if there exists a covering family of elliptic curves $\mathcal C = (C_t)_{t \in T}$  such that the canonical morphism
$$ \mathcal C \times_T  \mathcal C \to X \times X$$
is surjective. 
\end{dfn} 

\begin{rem} \label{rem2}   
(1) Gounelas \cite[Theorem 6.2]{Gou16} showed that an elliptically connected projective manifold is either rationally connected or admits a rationally connected fibration over
an elliptic curve.

(2) If $X$ is chain connected with respect to a family $(C_t)$ of curves, then nevertheless each class $[C_t]$ might be on the boundary of the movable cone of curves $\NMb(X)$, see \cite[Example 8.8]{BDPP}. However, each $[C_t]$ is in the interior of $\NEb(X)$ by \cite[Theorem 2.6]{BCE+}. 
\end{rem} 

We now show that in contrast to connectedness concepts for rational curves on projective manifolds, elliptic chain connectedness is a weaker property than elliptic connectedness.

\begin{thm} \label{thm:map} 
Let $X$ be a $\Q$-factorial projective threefold with $\rho(X) = 1$ and with terminal singularities. Assume that $K_X\sim0$ and that there exists a finite surjective map $f\colon X \to Y$ to a projective manifold $Y$ which contains a smooth rational curve $C$ with trivial normal bundle. Then $X$ is elliptically chain connected, but not elliptically connected. 
\end{thm} 

\begin{proof} 
First, since $X$ is not uniruled, it is not elliptically connected by \cite[Theorem 6.2]{Gou16}. 

It remains to show that $X$ is elliptically chain connected. It is a classical fact that the deformations of $C$ cover the whole variety $Y$, hence we obtain a covering family of rational curves $(C_s)_{s \in S}$ of $Y$. For each $s\in S$, let $\widetilde C_s $ be an irreducible component of $f^{-1}(C_s)$.  A general curve $C_s$ does not meet the image of the singular locus of $X$ (which is a finite set) by the morphism $f$, hence $f^{-1}(C_s)$ is contained in the smooth locus of $X$. Therefore, by an argument analogous to that of \cite[Lemme 2.2.3]{Ame04}, for a general $s\in S$ the curve $\widetilde C_s$ is smooth and has trivial normal bundle. Thus, $\widetilde C_s$ is an elliptic curve by the adjunction formula since $K_X$ is trivial and moreover, we obtain a covering family $\widetilde{\mathcal C}=(\widetilde C_s)_{s \in T}$ in $X$.  By Lemma \ref{lem:triv}, $X$ is $\widetilde{\mathcal C}$-chain connected. 
\end{proof} 

\begin{exa} \label{ex}  
It is easy to construct examples in the setting of Theorem \ref{thm:map}, see \cite[Theorem 1]{Cy03}. Let $Y$ be a Fano threefold of index $2$ and $\rho(Y) = 1$. Then $Y$ is covered by a family of lines whose general element has a trivial normal bundle, see \cite[Chapter 3]{IP99}. Take a smooth divisor $D \in | {-}2K_Y |$ and let $f\colon X \to Y$ the cyclic cover of degree $2$ attached to $D$. Then $X$ is a Calabi-Yau threefold with $h^{1,1}(X) = h^{1,1}(Y)$; furthermore, $\pi_1(X) = \pi_1(Y) = 0 $ by \cite{Cor81}, hence $\rho(X)=1$. By Theorem \ref{thm:map}, $X$ is elliptically chain connected, but not elliptically connected. 

Instead of considering a Fano threefold of index $2$, we might also take a Fano threefold of index $1$ and consider conics instead of lines. In any case, we obtain examples where $Y$ is not rational. 

We finally mention that Voisin \cite[Example 2.17]{Voi03} showed that any smooth double cover $X \to \PS^n$ is elliptically chain connected unless $X$ is of general type. 
\end{exa} 

In many cases the condition $\rho (X) = 1$ in Theorem \ref{thm:map} is redundant:

\begin{thm} \label{thm:map2} 
Let $X$ be a smooth Calabi-Yau threefold and let $f\colon X \to Y$ be a finite map to a smooth Fano threefold $Y$ with $\rho(Y) = 1$. Assume one of the following:
\begin{enumerate}
\item[(i)] $Y$ is of index $2$, and $H^3 = 3$ or $H^3 = 4$ for the ample generator $H$ of ${\Pic}(Y)$, i.e.\ $Y$ is a cubic hypersurface in $\PS^4$ or a complete intersection of two quadrics in $\PS^5$;
\item[(ii)] $Y$ is of index $1$ and $12 \leq -K_Y^3 \leq 18$. 
\end{enumerate} 
Then $X$ is elliptically chain connected. 
\end{thm} 

\begin{proof} 
As in Theorem \ref{thm:map}, we obtain a covering family $\mathcal C$ of elliptic curves and consider its quotient $g\colon  X \dasharrow Z$. Assume first that $\dim Z = 2$, so that $Z$ is a rational surface. Let $\mathcal H$ be the $2$-dimensional Hilbert scheme of lines on $Y$ (in the case (i)) respectively of conics (in the case (ii)). By construction, we obtain a rational dominant map $h\colon Z \dasharrow \mathcal H$, hence $\mathcal H$ is rational. This contradicts the classification of $\mathcal H$, see  \cite{KPS16}.

So it remains to rule out $\dim Z = 1$, i.e.\ $Z \simeq \PS^1$. Here the map $g$ is even holomorphic. By construction of the covering family $\mathcal C$, the general $C_t$ is smooth with trivial normal bundle. Consider a general smooth fibre $S$ of $g$: then $S$ is either a K3 surface or an abelian surface. The second alternative is excluded since  $S$ is elliptically chain connected. Thus, $S$ is a K3 surface which is chain connected by a family of elliptic curves whose general member $C_t$ is smooth with the trivial normal bundle in $S$. Since all $C_t$ are linearly equivalent, the linear system $|C_t|$ is basepoint free and defines an elliptic fibration $S \to \PS^1$. But then the family cannot be chain connecting in $S$. 
\end{proof} 

\begin{rem} 
In the context of the proof of the previous result, it is very likely that the irreducible components of the Hilbert scheme $\mathcal H$ are not rational in all other cases of Fano threefolds of index at most $2$ and Picard number $1$, with two exceptions. If $Y$ has index $2$ 
and $H^3 = 5$ (so that $Y\simeq V_5$ in the classical notation), or if $Y$ has index $1$ and ${-}K_Y^3 = 22 $ (so that $Y\simeq A_{22}$), then $\mathcal H \simeq \PS^2$, and we cannot conclude that $X$ is elliptically
chain connected.
\end{rem} 

\section{Elliptically chain connected varieties}\label{sec:ellchainconn}

In this section we study the structure of elliptically chain connected manifolds which are not uniruled. First observe:

\begin{rem} 
Any K3 surface $X$ is elliptically chain connected. Indeed, by \cite[Corollary 13.2.2]{Huy16}, $X$ contains a family of singular elliptic curves $C_t$ such that $C_t^2 > 0$. Therefore, the quotient of $X$ with respect to the family cannot be a fibration to a curve, hence the family must be chain connecting. 
\end{rem} 

On the other hand, it is clear that abelian varieties of dimension at least $2$ cannot be elliptically chain connected.

\begin{pro} \label{prop:etale} 
Let $X$ be a normal $\Q$-Gorenstein projective variety and let $\mathcal C = (C_t)_{t\in T} $ be a covering family of elliptic curves such that $X$ is smooth along the curve $C_t$ for general $t\in T$. Then:
\begin{enumerate} 
\item[(i)] $K_X \cdot C_t \leq 0$ for all $t$;
\item[(ii)] if $K_X$ is pseudoeffective, then the normalisation morphism of a general curve $C_t$ is \'etale and $K_X \cdot C_t = 0$ for all $t$;
\item[(iii)] if $K_X$ is pseudoeffective and if for a general $t\in T$ the curve $C_t$ is smooth, then its normal bundle $N_{C_t/X}$ is trivial.
\end{enumerate} 
\end{pro} 

\begin{proof} 
Possibly passing to a subfamily, we may assume that $$\dim T =  \dim X - 1.$$ 
If $\widetilde T$ and $\widetilde{\mathcal C}$ are the normalisation of $T$, respectively the normalised graph of $\mathcal C$, consider the projections $\widetilde p\colon \widetilde  {\mathcal C} \to X$ and  $\widetilde q\colon \widetilde  {\mathcal C} \to \widetilde T$. If $t \in \widetilde T$ is a general point, then $\widetilde C_t:=\widetilde q^{-1}(t)$ is a smooth elliptic curve which does not meet the singular locus of $\widetilde{\mathcal C}$ (notice that the singular locus of $\widetilde {\mathcal C}$ is of codimension at least $2$ and therefore disjoint from the general $\widetilde C_t$).  
We may write near $\widetilde C_t$:
\begin{equation}\label{eq:11}
K_{\widetilde{\mathcal C}} \sim \widetilde p^*K_X + E,
\end{equation}
where $E$ is an effective Weil divisor. Since $K_{\widetilde{\mathcal C}} \cdot \widetilde C_t = 0$ by the adjunction formula, the relation \eqref{eq:11} gives $K_X\cdot C_t=-E\cdot\widetilde C_t\leq0$, which proves (i) for general $t$, hence for all $t$ by \cite[Proposition I.3.12]{Kol96}. 

\medskip

For (ii), the assumption that $K_X$ is pseudoeffective and (i) immediately give
\begin{equation}\label{eq:22}
K_X \cdot C_t=0 \quad\text{for all }t.
\end{equation}
Note that we have the canonical sequence of cotangent sheaves outside of the singular locus of $\widetilde{\mathcal C}$: 
$$ 0 \to \widetilde p^*\Omega^1_X \to \Omega^1_{\widetilde {\mathcal C}} \to \Omega^1_{\widetilde {\mathcal C}/X} \to 0.$$
Thus  $(\widetilde p^*\Omega^1_X)|_{\widetilde C_t}$ is a subsheaf of $\Omega^1_{\widetilde {\mathcal C}}|_{\widetilde C_t} \simeq \OO_{\widetilde C_t}^{\dim X}$, and passing to determinants, we obtain an injective map $L\to \OO_{\widetilde C_t}$, where $L$ is numerically trivial on $\widetilde C_t$ by \eqref{eq:22}. Therefore, this last map is an isomorphism, and hence 
$$(\widetilde p^*\Omega^1_X)|_{\widetilde C_t} \simeq \Omega^1_{\widetilde {\mathcal C}}|_{\widetilde C_t}.$$
From the cotangent sequence above, this yields that $\widetilde p$ is \'etale in a neighbourhood of $\widetilde C_t$, and the normalisation map $\widetilde C_t \to C_t$ is \'etale as well.

\medskip

Suppose finally that the general curve $C_t$ is smooth. Since the deformations of $C_t$ cover $X$, we see -- passing to the graph -- that its normal bundle $N_{C_t/X}$ is generically globally generated. By adjunction, $\det N_{C_t/X}$ is numerically trivial by (ii), hence $N_{C_t/X}$ is trivial, which shows (iii).
\end{proof} 

\begin{cor} \label{cor:can}
Let $X$ be a normal projective variety with canonical singularities and let $\mathcal C = (C_t)_{t\in T} $ be a covering family of elliptic curves. Then 
\begin{enumerate} 
\item[(i)] $K_X \cdot C_t \leq 0$ for all $t$;
\item[(ii)] if $K_X$ is pseudoeffective, then $K_X \cdot C_t = 0$ for all $t$;
\item[(iii)] if $K_X$ is pseudoeffective and if $X$ has terminal singularities, then the general curve $C_t$ does not meet the singular locus of $X$.
\end{enumerate} 
\end{cor} 

\begin{proof} 
Let $\pi\colon X' \to X$ be a resolution of singularities. Then $X'$ carries a covering family $\mathcal C' = (C_t')_{t\in {T'}} $ of elliptic curves such 
that for general $t\in T'$, the curve $C_t'$ is the strict transform of $C_t$ in $X'$.
We may write 
$$ K_{X'} \sim_\Q \pi^*K_X + E, $$
where $E$ is an effective divisor. By Proposition \ref{prop:etale}(i) we have $K_{X'} \cdot C_t' \leq 0$, hence for a general $t\in T$ we have $K_X \cdot C_t = K_{X'} \cdot C_t' - E \cdot C_t\leq 0$, which shows (i), and (ii) follows immediately. 

For (iii), we have $K_{X'}\cdot C_t'=K_X \cdot C_t = 0$, and hence $E \cdot C_t = 0$. Therefore, a general curve $C_t$ cannot meet the singular locus of $X$.
\end{proof} 

\begin{thm} \label{thm:ecc1a} 
Let $X$ be a normal projective variety with canonical singularities such that $K_X$ is pseudoeffective and such that there exists a chain connecting family of elliptic curves $(C_t)_{t\in T}$.  Assume that $X$ has a minimal model via a sequence of divisorial contractions. Then $\kappa(X,K_X) = 0$. 
\end{thm} 

\begin{proof} 
Let $\varphi_i\colon X_i \to X_{i+1}$ be a sequence of divisorial contractions with $X_0=X$ such that $X_m$ is a minimal model of $X$ for some $m$, as in the statement of the theorem. We claim that for each $i\geq0$ there is a chain connecting family $(C_t^i)_{t\in T_i}$ of elliptic curves on $X_i$ such that $K_{X_i} \cdot C_t^i = 0$.

This immediately implies the theorem: indeed, then there exists a chain connecting family of curves $(C_t^m)_{t\in T_m}$ with $K_{X_m}\cdot C_t^m=0$. Then $K_{X_m} \equiv 0 $ by \cite[Theorem 2.6]{BCE+}, hence $K_{X_m}\sim_\Q0$ by \cite[Theorem 8.2]{Kaw85b}, and we conclude.
 
\medskip

It remains to show the claim. The claim holds for $i=0$ by Corollary \ref{cor:can}. Assume it holds for some $i\geq0$, and set $C_t^{i+1}:=\varphi(C_t^i)$. Then we obtain a chain connecting family of elliptic curves in $X_{i+1}$. If $E_i$ is the exceptional prime divisor of $\varphi_i$, then there is a positive rational number $\lambda_i$ such that 
\begin{equation}\label{eq:33}
K_{X_i} \sim_\Q \varphi_i^*K_{X_{i+1}} + \lambda_i E_i. 
\end{equation}
Since $K_{X_{i+1}}$ is pseudoeffective, we have $\varphi_i^*K_{X_{i+1}} \cdot C_t^i \geq 0 $ for a general $t\in T_i$, and we have $K_{X_i}\cdot C_t^i=0$ by induction. Then \eqref{eq:33} implies $E_i \cdot C_t^i \leq 0$, hence $ E_i \cdot C_t^i = 0$ since $C_t^i$ is not contained in $E_i$ for general $t\in T_i$. Therefore, $C_t^i\cap E_i=\emptyset$ and $K_{X_{i+1}}\cdot C_t^{i+1}=\varphi_i^*K_{X_{i+1}} \cdot C_t^i = 0$, which proves the claim.
\end{proof} 

\begin{rem} \label{rem1} 
We currently do not understand what happens if flips occur in the MMP. Let $X$ be as in the previous result. Let $\varphi\colon X\to Z$ be a flipping contraction and let $ \psi\colon X \dasharrow X^+$ be the associated flip. Let $ (p,q)\colon W \to X\times X^+$ be a resolution of indeterminacies of $\psi$. By \cite[Lemma 3.38]{KM98} there are $q$-exceptional divisors $E_i$ and non-negative rational numbers $a_i $ such that 
$$ p^*K_X = q^*K_{X^+} + \sum a_i E_i,$$
where $a_i>0$ if and only if $p(E_i)\subseteq\Exc(\varphi)$. Let $C_t'$ be the strict transform of $C_t$ in $W$, and let $C_t^+  = q(C_t')  \subseteq X^+$ be strict transform in $X^+$. Then
$$ q^*K_{X^+} \cdot C_t' \leq p^*K_X \cdot C_t',$$
with the strict inequality when $C_t\cap\Exc(\varphi)$ is non-empty and finite. On the other hand, since $K_X$ and $K_{X^+}$ are pseudoeffective, we have
$$K_X\cdot C_t= K_{X^+} \cdot C_t^+ = 0$$
by Corollary \ref{cor:can}(ii). Hence, we obtain the following alternative: either $C_t\cap\Exc(\varphi)=\emptyset$ or there is a component of $C_t$ which is contained in $\Exc(\varphi)$. However, it might possibly happen that the induced family $(C_t^+)_{t\in T^+}$ is no longer connecting. 
\end{rem} 

\begin{rem} \label{rem3}
If $X$ is elliptically chain connected with canonical singularities and $K_X$ is semiample, then $\kappa (X,K_X) = 0$. Indeed, let $\varphi\colon X \to Y$ be the associated Iitaka fibration, and consider a chain connecting family $\mathcal C = (C_t)$ of elliptic curves. Then each $C_t$ is contracted by $\varphi$ by Corollary \ref{cor:can}(ii), which can only happen if $Y$ is a point. 
\end{rem} 

\begin{rem} A weaker notion of elliptic chain connectedness would be to require that any two general points can be joined by a chain of elliptic and rational curves. Even then it is unclear whether we obtain a birationally invariant notion, except in dimension $3$ where the exceptional loci of flips are rational curves. 
\end{rem} 

\section{From elliptic curves to rational curves} \label{sec:elltorat}

Consider a smooth projective surface which is covered by elliptic curves, and assume that $X$ does not contain any rational curve. Then the classification of surfaces implies that $X$ is minimal, and it is either a torus, or hyperelliptic, or $\kappa (X,K_X) = 1$ and the Iitaka fibration is an almost smooth elliptic fibration, i.e.\ all singular fibres are multiples of elliptic curves. In particular, we have $c_2(X) = 0$. Conversely, none of these surfaces possesses a rational curve.
 
Here we study to which extent this picture generalises to higher dimensions. 

\begin{dfn} 
Let $X$ be a projective manifold and let $\varphi\colon X \to Y$ be a morphism with connected fibres to a normal projective variety. Then $\varphi$ is an \emph{almost smooth elliptic fibration} if all smooth fibres are elliptic curves and the underlying reduced scheme of any fibre is smooth. 
\end{dfn} 

Clearly the fibres of an almost smooth elliptic fibration do not contain rational curves. 

We begin with some preparations. Recall that a curve $C$ on a surface $S$ is called \emph{exceptional} if there is a birational map from $S$ to an analytic space which contracts $C$ to a point. 

\begin{lem} \label{lem:surface} 
Let $S$ be a normal projective surface and let $h\colon S \to B$ be a morphism with connected fibres to a smooth curve $B$. Assume that a general fibre of $h$ is an elliptic curve and that there are no rational curves in the fibres of $h$. Then $S$ is smooth and $h$ is an almost smooth elliptic fibration. Moreover, $h$ is isotrivial and $S$ does not contain any exceptional curve.
\end{lem} 

\begin{proof} 
If $\pi\colon \widehat S \to S$ is the minimal resolution, then the induced fibration $\widehat h\colon \widehat S \to B$ is relatively minimal. Since all fibres of $\widehat h$ contain non-rational curves, Kodaira's classification of the fibres of $\widehat h$ implies that $\widehat h$ does not contract any rational curve. So $\pi$ is an isomorphism and all fibres of $h$ are multiples of smooth elliptic curves. 

Then $h$ is isotrivial: indeed, by the stable reduction theorem \cite{KKMS} there is a finite base change $B'\to B$ and an induced fibration $h'\colon S'\to B'$ such that all fibres of $h'$ are smooth. Then $h'$ is locally trivial by \cite[Theorems III.17.3 and III.18.2]{BHPV04}, hence by \cite[Proposition VI.8]{Bea96}, after an \'etale base change, we may assume that $S'=B'\times E$ and $h'$ is the first projection, where $E$ is an elliptic curve. 

Finally, if $C\subseteq S$ is an exceptional curve on $S$, then $C^2<0$, and hence there exists a component $C'$ of the preimage of $C$ in $S'$ with 
\begin{equation}\label{eq:negative}
(C')^2<0. 
\end{equation}
However, note that any translation on $E$ induces a translation of $C'$ in $S'$. Then either $C'$ is a fibre of $h'$, which contradicts \eqref{eq:negative}, or these translations cover $S'$. But \eqref{eq:negative} implies that $C'$ does not move in its linear system, a contradiction.
\end{proof}

\begin{lem} \label{lem:flat-elliptic-a} 
Let $X$ and $Y$ be smooth projective varieties and let $q\colon X \to Y$ be an equidimensional surjective morphism such that the reduction of any fibre is a smooth elliptic curve. Then $q_*\OO_X(K_{X/Y})$ is a torsion line bundle on $Y$.
\end{lem} 

\begin{proof} 
First notice that 
$$ \mathcal L := q_*\OO_X(K_{X/Y}) $$
is locally free, and  of rank $1$: Indeed, $ q_*\OO_X(K_{X/Y}) $ is reflexive (by checking Hartogs-Riemann's extension theorem and using the equidimensionality of $q$), hence locally free, and it is of rank $1$ since a general fibre of $q$ is an elliptic curve.

Let $C\subseteq Y$ be a general complete intersection curve, and denote $X_C:=q^{-1}(C)$. Then
$$ \mathcal L|_C = q_*\OO_X(K_{X/Y})|_C = (q|_{X_C})_*(K_{X_C/C}). $$
Since $X_C$ is smooth, the line bundle $\mathcal L|_C$ is torsion by \cite[Theorem III.18.3]{BHPV04}. In particular,
\begin{equation}\label{eq:123}
\mathcal L \equiv 0.
\end{equation}

To show that $\mathcal L$ is torsion, we proceed by induction on $n=\dim Y$. The case $\dim Y = 1$ is already done. Let $H$ be a very ample divisor on $Y$ and $D \in | H |$ general. Using the induction hypothesis, we may choose a positive integer $m$ such that $H^0(D, \mathcal L^{\otimes m}|_D) \ne 0$. Now, as $\mathcal L^{\otimes -m} \otimes \OO_Y(D)$ is ample by \eqref{eq:123}, the restriction map
\begin{equation}\label{eq:124}
H^0(Y,\mathcal L^{\otimes m}) \to H^0(D, \mathcal L^{\otimes m}|_D)\quad\text{is surjective}
\end{equation}
since we have
$$ H^1\big(Y,\mathcal L^{\otimes m} \otimes \OO_Y(-D)\big)\simeq H^{n-1}\big(Y, \mathcal L^{\otimes -m} \otimes \OO_Y(K_Y+D)\big)  = 0$$
by the Serre duality and Kodaira vanishing. Thus $\mathcal L^{\otimes m} \simeq \OO_Y$ by \eqref{eq:124}. 
\end{proof} 

We need the following slight generalisation. 

\begin{lem} \label{lem:flat-elliptic} 
Let $X$ be a normal projective variety, let $Y$ be a projective manifold and let $q\colon X \to Y$ be an equidimensional surjective morphism. Assume furthermore:
\begin{enumerate} 
\item there exists a finite set $S \subseteq Y$ such that, if we set $Y_0:=Y \setminus S$ and $X_0:=X \setminus q^{-1}(S)$, the reduction of any fibre of the map $q|_{X_0}\colon X_0 \to Y_0$ is a smooth elliptic curve;
\item $X_0$ is smooth; 
\item $q$ does not contract any rational curve. 
\end{enumerate} 
Then $q_*\OO_X({K_{X/Y}})$ is a torsion line bundle on $Y$. 
\end{lem} 

\begin{proof}
As in the proof of Lemma \ref{lem:flat-elliptic-a} the sheaf
$$ \mathcal L := q_*\OO_X(K_{X/Y}) $$
is locally free of rank 1. Let $H \subseteq Y$ be a general hyperplane section. Since $q^{-1}(H)$ is smooth, $\mathcal L\vert_H$ is a torsion line bundle by Lemma \ref{lem:flat-elliptic-a}, and hence $\mathcal L \equiv 0$. We argue as in the proof of Lemma \ref{lem:flat-elliptic-a} to conclude that $\mathcal L$ is torsion. 
\end{proof} 

The next proposition is the technical heart of Theorem \ref{thm:CYHK}.

\begin{pro} \label{new1} 
Let $X$ be a projective manifold of dimension $n$  with a covering  family $\mathcal C$ of elliptic curves and let $\widetilde {\mathcal C}$ be the normalisation of $\mathcal{C}$. If $X$ does not contain a uniruled divisor, then: 
\begin{enumerate}
\item the family $\mathcal C$ has no moduli;
\item the projection $\widetilde p\colon \widetilde {\mathcal C} \to X$ is finite;
\item there exists a normal variety $Y$ and a finite surjective morphism $f\colon Y \to X$ such that 
$$H^0(Y,f^*T_X) \neq 0.$$
\end{enumerate} 
\end{pro} 

\begin{proof}  
\emph{Step 1.}
Let $q\colon \mathcal C \to T$ be the projection to the parameter space $T$.  Assume that the family of elliptic curves has moduli. Then the $j$-invariant yields a non-constant holomorphic map $j\colon T \to \PS^1$. Consequently, $q^{-1} \big(j^{-1}(\infty)\big)$ is a uniruled divisor in $ \widetilde {\mathcal C},$ projecting onto a uniruled divisor in $X$, see the proof of \cite[Corollary 3.34]{Voi03}. This is a contradiction which shows (1).
 
 \medskip
 
\emph{Step 2.} 
Suppose that $\widetilde p$ contracts a curve $\Gamma$. Let $\widetilde q\colon \widetilde { \mathcal  C} \to  \widetilde T $ denote the normalised projection, and set $\Gamma' := \widetilde q(\Gamma)$ and $S:=\widetilde q^{-1}(\Gamma')$. Since $\widetilde {\mathcal C}$ is the normalised graph, we have $\dim \widetilde p(S) = 2$, hence $\Gamma$ is a contractible curve in the surface $S$. Since the family $\mathcal C$ has no moduli by (1), Lemma \ref{lem:surface} applies to the Stein factorisation of the morphism $\widetilde{q}|_S\colon S\to \Gamma'$ and yields a contradiction. This shows (2). Note that we did not assume 
that $\dim T = n-1$; this is a consequence of our assumptions. 

\medskip

\emph{Step 3.} 
From now on we prove (3).  We will perform a weak semistable reduction to handle the multiple fibres of $\widetilde q$. To do that we need to pass to a suitable birational model as follows.

Let $\pi\colon W\to \widetilde T$ be a resolution of singularities and consider the base change: 
$$
\xymatrix{
Z \ar[d]_{\tau} \ar[r]^{\sigma} & {\widetilde {\mathcal C}}  \ar[d]^{\widetilde q}  \\
W \ar[r]_{\pi} & \widetilde T, &
}
$$
where $Z$ is the normalisation of the main component of $ \widetilde {\mathcal C} \times_{\widetilde T} W$. Notice that a general fibre of $q'$ is an elliptic curve and that there are no rational curves in the fibres of $q'$. 
 
Let $C \subseteq W$ be a general complete intersection curve. Then $Z_C := \tau^{-1}(C) $ is a normal surface, hence smooth by Lemma \ref{lem:surface}, and $\tau|_{Z_C}$ is almost smooth. Thus, $Z$ is smooth near $Z_C$ by Bertini's theorem. 

We claim that there exists a finite set $\Sigma \subseteq W$ such that, setting $W_0 = W \setminus \Sigma$ and $Z_0 = \tau^{-1}(W_0)$, the variety $Z_0$ is smooth, and the reduction of any fibre over $W_0$ is a smooth elliptic curve, i.e.\ $\tau|_{Z_0}$ is almost smooth by Lemma \ref{lem:surface}. Indeed, for each $w\in W$ let $Z_w$ be the scheme-theoretic fibre of $\tau$ over $w$. Set 
$$\Sigma:=\{w\in W\mid \Sing(Z)\cap Z_w\neq\emptyset\text{ and }(Z_w)_{\text{red}} \text{ is not smooth elliptic}\}.$$
If $\dim \Sigma\geq1$, then a general complete intersection curve $C$ cuts $\Sigma$. Since $Z_C$ is smooth and $\tau|_{Z_C}$ is almost smooth by above, this is a contradiction which proves the claim. 

By Lemma \ref{lem:flat-elliptic},  $\tau_*\OO_{Z}(K_{Z/W})$ is a torsion line bundle on $W$. After passing to a finite \'etale cover of $W$, we may assume that 
\begin{equation}\label{eq:125}
\tau_*\OO_{Z}(K_{Z/W}) \simeq \OO_{W}.
\end{equation}

\medskip

\emph{Step 4.}
We are now in a position to perform a weak semistable reduction as in \cite[Definition 6.7 and Lemma 6.8]{Vi95}. First of all, we choose a resolution $\widehat \pi \colon \widehat W \to W $ and a commutative diagram 
$$
\xymatrix{
\widehat Z  \ar[d]_{\widehat \tau} \ar[r]^{\widehat \sigma} & {Z}  \ar[d]^{\tau}  \\
\widehat W \ar[r]_{\widehat \pi} & W, &
}
$$
where $\widehat Z$ is a resolution of the main component $Z \times_{W} \widehat W$, such that there exists a divisor $\widehat \Sigma$ with simple normal crossings in $\widehat W$ with the property that $\widehat \tau^{-1}\big(\widehat \Sigma\big)$ is a divisor with simple normal crossings in $\widehat W$ and that $\widehat \tau|_{\widehat{Z}\setminus\widehat\tau^{-1}(\widehat{\Sigma})}$ is smooth. Let  $\widehat E_i$ denote the $\widehat \pi$-exceptional divisors. By \eqref{eq:125}, over $\widehat W \setminus \bigcup_i \widehat E_i$ we have an isomorphism
$$\widehat \tau_*\OO_{\widehat Z}\big(K_{\widehat Z / \widehat W}\big) \to \widehat \pi^* \tau_* \OO_Z(K_{Z/W}) \simeq \OO_{\widehat W}.$$
Thus
\begin{equation} \label{eq:int}  
\textstyle \widehat \tau_*\OO_{\widehat Z}\big(K_{\widehat Z/ \widehat W}\big) = \OO_{\widehat W}\big(\sum a_i \widehat E_i\big) 
\end{equation} 
with $a_i \in \Z$. 

By \cite[Lemma 6.8]{Vi95} there exist a commutative diagram 
$$
\xymatrix{
Z' \ar[d]_{\tau'} \ar[r]^{\sigma'} & {\widehat Z}  \ar[d]^{\widehat \tau}  \\
W' \ar[r]_{\pi'} & \widehat W &
}
$$
and a closed subset $\Pi\subseteq W'$ of codimension at least $2$ such that $W'$ is smooth, $Z'$ is normal, $\pi'$ is finite, $\sigma' $ is generically finite and the fibres of $\tau'$ are smooth elliptic curves outside $\Pi$. Moreover, by \cite[Lemma 6.9]{Vi95} the morphism $\tau'$ is flat over $W' \setminus \Pi$ and the line bundle $\tau'_*\OO_{Z'}(K_{Z'/W'})$ satisfies 
\begin{equation}\label{eq:127}
\textstyle  \tau'_*\OO_{Z'}(K_{Z'/W'}) \subseteq \pi'^*\widehat\tau_*\OO_{\widehat Z}\big(K_{\widehat Z/ \widehat W}\big) = \OO_{W'}\big(\sum a_i \pi'^*\widehat E_i\big),
\end{equation}
where the last equality follows by \eqref{eq:int}. By \cite[Theorem 6.12]{Vi95} the line bundle $\tau'_*\OO_{Z'}(K_{Z' / W'})$ is nef. 

We claim
\begin{equation}\label{eq:126}
\tau'_*\OO_{Z'}(K_{Z' / W'})\simeq\OO_{W'}.
\end{equation}
Indeed, let $L$ be a nef Cartier divisor such that $\OO_{W'}(L)\simeq\tau'_*\OO_{Z'}(K_{Z' / W'})$, let $\lambda\colon W'\to W''$ be the birational part of the Stein factorisation of the generically finite map $\widehat \pi \circ \pi'$, and denote $E_i':=\pi'^*\widehat E_i$. Then the divisors $E_i'$ are $\lambda$-exceptional. By \eqref{eq:127} we may assume that there exists an effective Cartier divisor $M$ such that
$$\textstyle L+M=\sum a_i E_i'.$$
By the Negativity lemma \cite[Lemma 3.39]{KM98}, applied to $\lambda$, we have $M\geq \sum a_i E_i'$, and hence $L\leq0$. But then $L=0$ and the claim \eqref{eq:126} follows.

\medskip 

\emph{Step 5.}
We define the generically finite morphism $h'\colon Z'\to X$ by the following commutative diagram:
$$
\xymatrix{
Z' \ar[d]_{\tau'} \ar[r]_{\sigma'} \ar@/^1.2pc/[rrrr]^{h'} & \widehat Z  \ar[d]_{\widehat \tau} \ar[r]_{\widehat \sigma} & {Z}  \ar[d]_{\tau} \ar[r]_\sigma & \widetilde{\mathcal{C}} \ar[d]_{\widetilde{q}} \ar[r]_{\widetilde{p}} & X  \\
W' \ar[r]_{\pi'} & \widehat W \ar[r]_{\widehat \pi} & W \ar[r]_{\pi} &  \widetilde{T}. &
}
$$
Then we obtain a non-zero morphism
\begin{equation}\label{eq:128}
h'^*\Omega^1_X  \to \Omega^1_{Z'} \to \Omega^1_{Z'/W'} \to \OO_{Z'}(K_{Z'/W'}). 
\end{equation}
The canonical injective morphism
$$ \tau'^* \tau'_*\OO_{Z'}(K_{Z'/W'}) \to \OO_{Z'}(K_{Z'/W'})$$
is an isomorphism away from $\tau'^{-1}(\Pi)$, since the fibres of $\tau'$ outside $\Pi$ are smooth elliptic curves.  Since $\tau'^* \tau'_*\OO_{Z'}(K_{Z'/W'}) \simeq \OO_{Z'}$ by \eqref{eq:126}, there exists an effective divisor $D'$ supported on $ \tau'^{-1}(\Pi)$ such that 
\begin{equation}\label{eq:129}
\OO_{Z'}(K_{Z'/W'}) = \OO_{Z'}(D').
\end{equation}
Hence, by \eqref{eq:128} and \eqref{eq:129} we obtain
\begin{equation}\label{eq:131}
H^0\big(Z',h'^*T_X \otimes \OO_{Z'}(D')\big) \neq 0.
\end{equation}
Since $\widetilde q$ is equidimensional and since $\codim_{\widetilde{T}} (\pi \circ \widehat \pi \circ \pi')(\Pi) \geq 2$, we conclude that  
\begin{equation}\label{eq:130}
\codim_X h'(D') \geq 2.
\end{equation}
Finally, let 
$$ Z' \buildrel h_1' \over {\longrightarrow}  Y \buildrel f \over {\longrightarrow} X $$
be the Stein factorization of $h'$. 
Then 
$$H^0(Y, f^*T_X) \neq 0$$
by \eqref{eq:131} and \eqref{eq:130}.
\end{proof} 

The following is the main result in this section. 

\begin{thm}  \label{thm:CYHK} 
Let $X$ be a Calabi-Yau or a projective hyperk\"ahler manifold. If $X$ has a covering family of elliptic curves, then $X$ contains a uniruled divisor. 
\end{thm} 

\begin{proof} 
Arguing by contradiction and applying Proposition \ref{new1}, there exists a normal projective variety $Y$ and a finite surjective map $f\colon Y \to X$ such that
$$ H^0(Y, f^*T_X) \neq 0.$$ 
Since the map $f$ is finite, there exists a reflexive sheaf $\sF$ such that 
$$f_*\sO_{Y} \simeq \sO_X \oplus \sF.$$
Consequently, 
\begin{equation}\label{eq:sum1}
0\neq H^0(X,f_*f^*T_X) = H^0(X, T_X)\oplus H^0(X,T_X \otimes \sF).
\end{equation}
Since $X$ is a Calabi-Yau or a hyperk\"ahler manifold, we have $H^0(X,T_X) = 0$, which in combination with \eqref{eq:sum1} yields
$$ H^0(X, T_X \otimes \sF) \ne 0.$$ 
Therefore, there exists a non-zero morphism
$$ \alpha\colon \sF^* \to T_X.$$ 
Let $H$ be an ample divisor on $X$ and let 
$$S  = H_1 \cap \ldots \cap H_{n-2}$$ 
be a surface cut by general hyperplane sections $H_j \in | mH | $ of degree $m \gg 0$. Set $Y_S := f^{-1}(S)$. Then $Y_S$ normal and therefore Cohen-Macaulay. In particular, $f|_{Y_S}\colon Y_S \to S$ is flat. By a result of Lazarsfeld \cite[Appendix, Proposition A]{PS00}, applied to the flat morphism $f_S$, the locally free sheaf 
$$\big(\big((f|_{Y_S})_*(\OO_{Y_S})\big)^*/\OO_S\big)|_C$$ 
is nef for any curve $C \subseteq S$ which is not contained in the branch locus $B$ of $f|_S$. But this sheaf is precisely $\mathcal F^*|_C$ by definition. Consequently, the sheaves $(\alpha|_S)(\mathcal F^*|_S){|_C}$ and $(\det (\alpha|_S)(\mathcal F^*|_S)){|_C}$ are nef. In particular, this holds for any general curve on $S$ cut out by $mH|_S$. Since $T_X$ is $H$-stable, the restriction $T_X|_S$ is $H|_S$-stable by the theorem of Mehta-Ramanathan, hence
$$ \rk(\alpha|_S)(\mathcal F^*|_S)=\dim X.$$ 
Thus, there exists a proper closed subset $A\subseteq S$ such that $\alpha|_S$ is surjective on $S \setminus A $. Since $(\alpha|_S)(\mathcal F^*|_S){|_C}$ is nef on every curve $C \not \subseteq B \cup A$, so is ${T_X}|_C$. 

To summarise, the locally free sheaf
$$ \mathcal E := T_X|_S$$ 
is $H|_S$-stable with $c_1(\mathcal E) = 0$ and is nef on all but finitely many curves. By \cite[Corollary 5.4]{HP17}, $\mathcal E$ is numerically flat,\footnote{Actually, $\mathcal{E}$ is even trivial, since it has a filtration by unitary flat bundles and since $\pi_1(S) = 0$.} and in particular, $c_2(\mathcal E) = 0$. Thus 
$$ c_2(X) \cdot H^{n-2} = 0,$$ 
and consequently,
$$c_2(X) = 0.$$
By Yau's theorem, $X$ is then an \'etale quotient of a torus, which contradicts our assumption that $X$ is a Calabi-Yau or a hyperk\"ahler manifold. 
\end{proof} 

The proof  of Theorem \ref{thm:CYHK} shows also the following. 

\begin{thm} 
Let $X$ be a projective manifold which contains a covering family of elliptic curves. Then $X$ contains a uniruled divisor unless there exists a torsion free quotient $Q$ of $ \Omega^1_X $ such that $\det Q \equiv 0$.
\end{thm} 

\begin{proof} 
We may assume that $X$ itself is not uniruled. If $\Omega^1_X $ is generically ample, i.e., ample on curves cut out by general hyperplane sections, then the proof of Theorem \ref{thm:CYHK} applies. If $\Omega^1_X $ is not generically ample, then \cite[Proposition 2]{Pe11} gives the quotient $Q$ as required. 
\end{proof} 

If $X$ is elliptically chain connected, then the existence of rational curves follows without any further assumptions on $X$: 

\begin{thm} \label{thm:ellipticrational}
Let $X$ be an elliptically chain connected projective manifold with $\dim X \geq 2$. Then $X$ contains a rational curve. 
\end{thm} 

\begin{proof} 
Arguing by contradiction, assume that $X$ does not contain a rational curve. Then $K_X$ is nef by Mori's Cone theorem, hence $K_X \sim_\Q0$ by Theorem \ref{thm:ecc1a}. By the Beauville-Bogomolov decomposition theorem, there exists a finite \'etale cover from a product of Calabi-Yau manifolds, hyperk\"ahler manifolds and an abelian variety to $X$. We may assume that $X$ itself decomposes as
$$ X = X_1 \times \ldots \times X_r.$$

If  $r = 1$,  then $X$ must be Calabi-Yau or hyperk\"ahler, since abelian varieties of dimension at least two are not elliptically chain connected. Then we conclude by Theorem \ref{thm:CYHK}.

If $r \geq 2,$ then all factors $X_j$ are elliptically chain connected by Lemma \ref{lem:conn}. Furthermore, at most one factor can be abelian and this factor has to be of dimension one. Hence one of the factors $X_j$ is Calabi-Yau or hyperk\"ahler and we conclude again by Theorem \ref{thm:CYHK}.
\end{proof} 

\begin{lem} \label{lem:conn} 
Let $X$ be an elliptically chain connected projective manifold. Assume that $X \simeq Y_1 \times Y_2$, where $Y_1$ and $Y_2$ are projective manifolds. Then $Y_1$ and $Y_2$ are elliptically chain connected. 
\end{lem} 

\begin{proof} 
Let $(C_t)_{t\in T}$ be a chain connecting family of elliptic curves on $X$. By symmetry, it suffices to show that $Y_1$ is elliptically chain connected, and let $p_1\colon X \to Y_1$ denote the first projection. 

Suppose that $\dim p_1(C_{t_0}) = 0$ for some $t_0\in T$. If $H$ is an ample divisor on $Y_1$, then $p_1^*H\cdot C_{t_0}=0$, hence $p_1^*H\cdot C_t=0$ for all $t\in T$ by \cite[Proposition I.3.12]{Kol96}. Therefore, $\dim p_1(C_t) = 0$ for all $t\in T$, i.e., all curves $C_t$ are contained in fibres of $p_1$. But then the family $(C_t)_{t\in T}$ cannot be chain connecting. 

Thus $\dim p_1(C_t) = 1,$ and therefore we obtain a covering family $(C'_s)$ of elliptic curves on $Y_1$. This family is chain connecting: choose general points $a_1, a_2 \in Y_1$ and general points $b_1,b_2 \in Y_2$. Then the points $(a_1,b_1),(a_2,b_2)\in Y_1\times Y_2$ can be joined by a chain of curves $C_t$, hence $a_1$ and $a_2$ can be joined by a chain of curves $C'_s$. 
\end{proof} 

\begin{thm} \label{thm:kappa0} 
Let $X$ be a smooth projective variety with $\kappa (X,K_X) = 0$. Assume that $X$ carries a covering family of elliptic curves and that $X$ does not contain any rational curve. Then there is a finite \'etale cover $X' \to X$ with 
$$X'=E \times \prod\limits_j X_j,$$ 
where $E$ is an elliptic curve and each $X_j$ is either a torus, a Calabi-Yau manifold or a hyperk\"ahler manifold. 
\end{thm} 

\begin{proof} 
As in the proof of Theorem \ref{thm:ellipticrational}, there is a finite \'etale cover $X'\to X$ such that
$$X'= T \times \prod_j X_j,$$
where $T$ is a torus (a priori possibly of dimension $0$) and each $X_j$ is a Calabi-Yau manifold or a hyperk\"ahler manifold. Considering a covering family $(C'_s)$ of $X'$ and the induced maps $C'_s \to X_j$ and $C'_s \to T$, we conclude that at least one of the factors $X_j$ or $T$ is covered by elliptic curves. In the first case, $X_j$ would contain a rational curve by Theorem \ref{thm:CYHK}, a contradiction. Thus $\dim T > 0 $ and $T$ is covered by a family of elliptic curves. Hence, by Poincar\'e's Reducibility theorem, after passing to a finite \'etale cover, we may assume that $T = T' \times E$, where $E$ is an elliptic curve. 
\end{proof} 

\begin{rem} 
In fact, in Theorem \ref{thm:kappa0} we expect that $X'$ is a product of an elliptic curve and a torus. This, however, requires to prove that Calabi-Yau and hyperk\"ahler manifolds contain rational curves, which seems out of reach at the moment. 
\end{rem}

\begin{thm} \label{thm:ratcurves} 
Let $X$ be a projective manifold of dimension at least $2$ with a covering family $\sC$ of elliptic curves. Suppose $X$ does not contain rational curves. Then there exists an equidimensional fibration $\varphi\colon X \to W$ to a normal projective variety $W$ with the following properties:
\begin{enumerate} 
\item[(i)] $\varphi$ contract all elements of $\mathcal C$; more precisely, $\varphi$ is the $\mathcal C$-quotient of $X$;
\item[(ii)] all fibres of $\varphi$ are irreducible;
\item[(iii)] the normalisation of any fibre of $\varphi$ is an elliptic curve;
\item[(iv)] $\varphi$ is an almost smooth elliptic fibration over the smooth locus of $W$;
\item[(v)] $W$ has klt singularities.
\end{enumerate} 
\end{thm} 

\begin{proof}  
\emph{Step 1.}
By Theorem \ref{thm:ellipticrational}, the manifold $X$ is not elliptically chain connected with respect to $\mathcal C$. Let $\varphi\colon X \dasharrow W$ be the $\mathcal C$-quotient of $X$, and recall that $\varphi$ is almost holomorphic. Assume the dimension of a general fibre $F$ of $\varphi$ is at least $2$. Since $F$ is elliptically chain connected, Theorem \ref{thm:ellipticrational} implies that $F$ possesses a rational curve, a contradiction. Therefore, a general fibre of $\varphi$ is a curve.

Let $q\colon \widetilde \sC \to T$ be the normalized graph of the family with projection $p\colon \widetilde \sC \to X$, and for each $t\in T$ denote $C_t:=p(q^{-1}(t))$. Since there is a unique elliptic curve through a general point $x \in X$, the morphism $p$ is birational, and we may assume that $W = T$ and $\varphi=q\circ p^{-1}$. To show that $\varphi$ is a morphism, it suffices to show that $p$ is an isomorphism. By Zariski's main theorem, it suffices to show that $p$ is finite.

\medskip

\emph{Step 2.}
Assume to the contrary that $\dim p^{-1}(x) > 0$ for some $x \in X$, and set
$$T(x) = \{t \in T \mid x \in C_t \}.$$
Then $\dim T(x) > 0$; choose an irreducible curve $C \subseteq T(x)$. Let $\sigma\colon \widehat T \to T$ be a desingularisation 
and $\widehat \sC$ be the normalisation of the main component of fibre product $\widetilde \sC \times_T \widehat T$, so that we have a diagram
$$
\xymatrix{
\widehat{\sC} \ar[r]^{\tau} \ar[d]_{\widehat q} & \widetilde \sC \ar[d]^{q} \\
\widehat T \ar[r]_{\sigma} & T.
}
$$
Possibly by blowing up further, we may assume that there exists a smooth curve $\widehat C \subseteq \widehat T$ such that $\sigma (\widehat C) = C$; choose an irreducible component $\widehat S$ of $\widehat q^{-1}(\widehat C)$ with normalisation $\nu\colon\widehat{S}^\nu \to \widehat S$. Let 
$$f\colon \widehat S^\nu \to \widehat C$$ 
be the induced map and notice that the fibres of $f$ do not contain rational curves, since otherwise $q$ would contain a rational curve in a fibre. 

We claim that a general fibre of $f$ is an elliptic curve. Indeed, pick a general point $y \in \widehat C$, and let $B$ be a general smooth curve containing $y$ such that the surface $S_B := \widehat q^{-1}(B)$ is normal and such that a general fibre of $S_B$ is an elliptic curve. Then Lemma \ref{lem:surface} implies that $S_B$ is smooth and the projection $S_B \to B$ is almost smooth. Hence, a general fibre of $f$ is likewise elliptic, which proves the claim.

Lemma \ref{lem:surface} now implies that $\widehat S^\nu$ is smooth, that $f$ is almost smooth, and that $\widehat S^\nu$ does not contain exceptional curves. However, the birational map $\widehat S^\nu \to X$ contracts $\nu^{-1}(\widehat S\cap (p\circ\tau)^{-1}(x))$, a contradiction. Thus $p$ is finite, hence an isomorphism. Consequently, $\varphi\colon X \to W$ is an equidimensional morphism.

\medskip

\emph{Step 3.}
The arguments above also show that $\varphi$ is almost smooth over the smooth locus of $W$. To see that each fibre of $\varphi$ is an irreducible curve whose normalisation is an elliptic curve, let $w \in W$ be a singular point. Let $\Gamma \subseteq W$ be a general irreducible curve through $w$. Then the preimage $S_\Gamma = \varphi^{-1}(\Gamma)$ is an irreducible surface in $X$. Let $h\colon S_\Gamma^\nu \to \Gamma^\nu$ be the induced map between normalisations. Then by Lemma \ref{lem:surface}, $h$ is an almost smooth elliptic fibration. Hence $\varphi^{-1}(w)$ is an irreducible curve whose normalisation is elliptic. 

\medskip

\emph{Step 4.}
We claim first that $K_X$ is numerically trivial on every fibre of $\varphi$. Indeed, this is clear for a general fibre, which is a smooth elliptic curve. If $F$ is a special fibre, then $F_{\textrm{red}}$ is irreducible. So if $K_X|_{F_{\textrm{red}}}$ were not numerically trivial, then $K_X|_{F_{\textrm{red}}}$ would be ample or anti-ample. Hence the same would be true for the nearby fibres, as ampleness is an open property, which is a contradiction which proves the claim. 

Therefore, the manifold $X$ has a relative good model over $W$ by \cite[Theorem 2.12]{HX13}, and in particular, $K_X\sim_\Q\varphi^*A$ for some $\Q$-Cartier divisor $A$ on $W$. Consequently, there exists an effective divisor $\Delta_W$ on $W$ such that $(W,\Delta_W)$ is klt by \cite[Theorem 0.2]{Amb05a}.
\end{proof} 

Using \cite{DFM16}, we conclude again that a Calabi-Yau or hyperk\"ahler manifold which admits a covering family of elliptic curves, contains a rational curve. 

\section{Open problems} 
In this section we address some conjectures and open problems. The first question already came up in Theorem \ref{thm:ecc1a} and Remark \ref{rem1}: 

\begin{prb} 
Is elliptic chain connectedness a birational property of terminal or klt varieties? 
\end{prb} 

Next, we recall the following:

\begin{dfn} 
A projective manifold $X$ is \emph{special} if for any $p>0$ and for any locally free subsheaf $\mathcal L \subseteq \Omega^p_X$ of rank $1$,  we have $\kappa (X,\mathcal L) \leq p-1$. 
\end{dfn}

This is not Campana's original definition, but an equivalent statement, see \cite[Theorem 2.27]{Ca04a}. Campana's theory yields the following  statement.

\begin{pro} 
Any torically chain connected projective manifold is special. 
\end{pro}

\begin{proof} 
This follows from \cite[Theorem 3.3(2) and Theorem 5.1]{Ca04a}.
\end{proof} 

\begin{prb} 
Therefore, given a torically chain connected projective manifold, we have $\kappa (X,  \mathcal L) \leq p-1$ for any  locally free subsheaf $\mathcal L \subseteq \Omega^p_X $ of rank $1$. Is it, actually, true that $\kappa (X, \mathcal L) \leq 0$? By Campana's theory, this is -- modulo the Minimal Model Program --  equivalent to saying that $\kappa (X,K_X) \leq 0$. 
\end{prb} 

Following \cite[Definition 3.26]{Voi03}, one may consider whether a variety is \emph{rationally swept out} by a family of varieties. In particular, to say that $X$ is rationally swept out by abelian varieties is equivalent to the existence of a covering 
family of subvarieties such that the general member has a finite cover which is birational to an abelian variety.  
Then we have the following version of Lang's conjecture  \cite[Conjecture 3.31]{Voi03}.

\begin{con} 
Let $X$ be a projective manifold with $0 \leq \kappa (X,K_X) \leq \dim X - 1$. Then $X$ is rationally swept out by abelian varieties.\end{con} 

One might ask for a higher-dimensional analogue of Theorem \ref{thm:ratcurves}:

\begin{prb}  
Let $X$ be a projective manifold with $0 \leq \kappa (X,K_X) \leq \dim X - 1$. Assume that $X$ is rationally swept out by abelian varieties, and suppose that $X$ has no rational curves. Does there exists a holomorphic equidimensional fibre space $\varphi\colon X \to W $ to a normal klt variety $W$ such that:
\begin{enumerate}
\item[(a)] all fibres are irreducible (but possibly non-reduced),
\item[(b)] the normalisation of any fibre is an abelian variety up to finite \'etale cover?
\end{enumerate} 
\end{prb} 

\begin{prb}  
Are there Calabi-Yau threefolds which are rationally swept out by tori  or even torically connected, but not by elliptic curves? As possible examples, one might consider Calabi-Yau threefolds with $\rho(X) = 2$ admitting a fibre space structure over $\PS^1$ whose general fibre is a simple abelian surface. A weaker question would be: Is there a Calabi-Yau threefold which is not elliptically chain connected (but possibly covered by elliptic curves)? 
\end{prb} 

\begin{rem}
If $X $ is a general hypersurface of degree $n+2$ in $\PS^{n+1}$, then $X$ is not rationally swept out by abelian varieties of dimension at least $2$, see \cite[Theorem 3.30]{Voi03}. Therefore, Lang's conjecture predicts that $X$ is covered by elliptic curves. However, as Voisin shows, this contradicts a conjecture of Clemens.
\end{rem}

\bibliographystyle{amsalpha}

\bibliography{biblio}

\end{document}